\documentclass[11pt]{amsart}
\usepackage{amsmath,amssymb,graphicx}
\usepackage{tikz} \usetikzlibrary{calc,backgrounds,arrows}

\newtheorem{thm}{Theorem}[section]
\newtheorem{main}{Theorem}\renewcommand{\themain}{\Alph{main}}
\newtheorem{prop}[thm]{Proposition}

\theoremstyle{definition}\newtheorem{defn}[thm]{Definition}
\newtheorem{exmp}[thm]{Example}\newtheorem{rem}[thm]{Remark}
\newtheorem{prob}[thm]{Problem}

\newcommand{\Z}{\mathbb{Z}}
\newcommand{\R}{\mathbb{R}}

\newcommand{\E}{\mathcal{E}}

\newcommand{\lc}{\textrm{lc}}
\newcommand{\rc}{\textrm{rc}}
\newcommand{\disc}{\mathbf{D}}
\newcommand{\vv}{\mathbf{v}}
\newcommand{\third}{\textrm{new}}

\newcommand{\sym}{\textsc{Sym}}\newcommand{\braid}{\textsc{Braid}}

\newcommand{\norm}{\textsc{Norm}}

\newcommand{\simp}{\textsc{Simp}}
\newcommand{\rot}{\textsc{Rot}}
\newcommand{\gen}{\textsc{Gen}}

\newcommand{\std}{\textsc{Std}}

\newcommand{\bi}{\begin{itemize}}\newcommand{\ei}{\end{itemize}}
\newcommand{\be}{\begin{enumerate}}\newcommand{\ee}{\end{enumerate}}
\newcommand{\bc}{\begin{center}}\newcommand{\ec}{\end{center}}
\newcommand{\bt}{\begin{tabular}}\newcommand{\et}{\end{tabular}}

\newcommand{\drawArrow}[2]{\draw[-triangle 60,thick,color=blue] #1--($#1!.85!#2$);}
\newcommand{\id}{1}

\tikzstyle{YellowPoly}=[thin,color=black,fill=yellow!20,join=bevel]
\tikzstyle{GreenPoly}=[thin,color=black,fill=green!20,join=bevel]
\tikzstyle{BlueLine}=[very thick,color=blue,join=bevel]
\tikzstyle{RedLine}=[very thick,color=red,join=bevel]
\tikzstyle{BluePoly}=[BlueLine,fill=blue!20]
\tikzstyle{RedPoly}=[RedLine,fill=red!20]
\tikzstyle{smallDot}=[draw,shape=circle,color=black,fill=black,inner sep=.75pt]
\tikzstyle{Rect}=[fill=yellow!20,rounded corners,minimum width=.9cm,minimum height=.9cm,draw]

\begin{document}

\title{Braid groups and euclidean simplices}
\author{Elizabeth Leyton Chisholm and Jon McCammond} 
\date{\today}


\begin{abstract}
When Daan Krammer and Stephen Bigelow independently proved that braid
groups are linear, they used the Lawrence-Krammer-Bigelow
representation for generic values of its variables $q$ and $t$.  The
$t$ variable is closely connected to the traditional Garside structure
of the braid group and plays a major role in Krammer's algebraic
proof.  The $q$ variable, associated with the dual Garside structure
of the braid group, has received less attention.

In this article we give a geometric interpretation of the $q$ portion
of the LKB representation in terms of an action of the braid group on
the space of non-degenerate euclidean simplices.  In our
interpretation, braid group elements act by systematically reshaping
(and relabeling) euclidean simplices.  The reshapings associated to
the simple elements in the dual Garside structure of the braid group
are of an especially elementary type that we call relabeling and
rescaling.
\end{abstract}
\maketitle

At the turn of the millenium three papers on the linearity of braid
groups appeared in rapid succession and all three used what is now
known as the Lawrence-Krammer-Bigelow or LKB representation
\cite{Kr00,Bi01,Kr02}. Its two variables, $q$ and $t$, are connected
to two different Garside structures on the braid group.  The $t$
variable is closely connected to the traditional Garside structure of
the braid group and plays a major role in Krammer's algebraic proof
\cite{Kr02}.  The $q$ variable is associated with the dual Garside
structure and has received less attention.  In this article, we
introduce an elegant geometric interpretation of the $q$ variable in
the special case where $t=1$, $q$ is real, and the matrices of the
representation are written with respect to the original basis used by
Krammer in \cite{Kr00}.  We call this special case the simplicial
representation because of our first main result.

\begin{main}[Braids reshape simplices]\label{main:braid}
  The simplicial representation of the $n$-string braid group
  preserves the set of $\binom{n}{2}$-tuples of positive reals that
  represent the squared edge lengths of a nondegenerate euclidean
  simplex with $n$ labeled vertices.
\end{main}

Thus braid group elements can be viewed as acting on and
systematically reshaping the space of all nondegenerate euclidean
simplices.  Moreover, the dual simple braids in the braid group
reshape simplices in an extremely elementary way that we call
relabeling and rescaling.  In this language we prove the following
result; for a more precise statement, see Section~\ref{sec:matrices}.

\begin{main}[Dual simple braids relabel and rescale]\label{main:simple}
  Under the simplicial representation of the braid group, each dual
  simple braid acts by relabeling the vertices and rescaling specific
  edges.
\end{main}

The article is structured around the three distinct contexts that play
a role in these results: discs, simplices and matrices.  In
Section~\ref{sec:discs} convexly punctured discs are used to define
noncrossing partitions and dual simple braids.  In
Section~\ref{sec:simplices} we describe various ways to systematically
reshape euclidean simplices, including the type of reshaping that we
call edge rescaling.  In Section~\ref{sec:matrices} we connect these
two relatively elementary discussions with the explicit matrices of
the simplicial representation of the braid group to establish our main
results.  Finally, Section~\ref{sec:remarks} explains our motivation
for pursuing this line of investigation and some ideas for future
work.

\section{Discs}\label{sec:discs}
In this section metric discs with a finite number of labeled points
are used to define the lattice of noncrossing partitions and the
finite set of dual simple braids.  We begin by recalling the notion of
a convexly punctured disc.

\begin{defn}[Convexly punctured disc]\label{def:discs}
  Let $\disc_n$ be a topological disc in the euclidean plane with a
  distinguished $n$-element subset that we call its \emph{punctures}
  or \emph{vertices}. When the disc $\disc_n$ is a convex subset of
  $\R^2$ and the convex hull of its $n$ punctures is an $n$-gon
  (i.e. every puncture occurs as a vertex of the convex hull) then we
  say that $\disc_n$ is a \emph{convexly punctured disc}. There is a
  natural cyclic ordering of the vertices corresponding to the
  clockwise orientation of the boundary cycle of the $n$-gon. A
  labeling of the vertices is said to be \emph{standard} if it uses
  the set $[n] := \{1,2,...,n\}$ (or better yet $\Z/n\Z$) and the
  vertices are labeled in the natural cyclic order.  More generally,
  when the vertices $p_i$ are bijectively labeled by elements $i$ in a
  finite set $A$, we refer to the convexly punctured disc as
  $\disc_A$.
\end{defn}

The $2$-elements subsets are of particular interest.

\begin{defn}[Edges]\label{def:edges}
  Let $\disc_n$ be a convexly punctured disc.  For each two element
  subset $\{i,j\} \subset [n]$, the convex hull of the corresponding
  points $p_i$ and $p_j$ in $\disc_n$ is called an \emph{edge} and
  denoted $e_{i,j} = e_{j,i}$, or even $e_{ij}$ when the comma is not
  needed for clarity.  When a \emph{standard name} is needed we insist
  $i < j$.  The number of edges is $\binom{n}{2}$ and we consistently
  use $N$ for this number throughout the article.  For later use, it
  is also convenient to impose a \emph{standard order} on the set of
  all $N=\binom{n}{2}$ edges.  We do so by lexigraphically ordering
  them by their standard names.  In $\disc_4$, for example, the
  standard names of its $6 = \binom{4}{2}$ edges in their standard
  order are $e_{12}$, $e_{13}$, $e_{14}$, $e_{23}$, $e_{24}$ and
  $e_{34}$.
\end{defn}

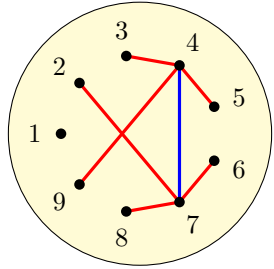
\begin{figure}
  \begin{tikzpicture}[scale=.7]
      \filldraw[YellowPoly] (0,0) circle (2.5cm);
      \foreach \x in {1,2,...,9} {\coordinate (\x) at (220-40*\x:1.5cm);}
      \draw[RedLine] (3)--(4)--(5) (4)--(9);
      \draw[RedLine] (2)--(7)--(6) (7)--(8);
      \draw[BlueLine] (4)--(7);
      \foreach \x in {1,2,...,9} {\fill (\x) circle (1mm);}
      \foreach \x in {1,2,...,9} {\draw (220-40*\x:2cm) node {\small \x};}
  \end{tikzpicture}
  \caption{The edges $e_{34}$, $e_{49}$ and $e_{67}$ are to the left
    of the edge $e_{47}$ and the edges $e_{27}$, $e_{78}$ and $e_{45}$
    are to the right.  This is because ordered pairs such as
    $(e_{34},e_{47})$ and $(e_{67},e_{47})$ are clockwise while the
    ordered pair $(e_{27},e_{47})$ is
    counterclockwise.\label{fig:left-right}}
\end{figure}

We also need words to describe the position of one edge relative to
another.

\begin{defn}[Pairs of edges]\label{def:edge-pairs}
  Let $(e_{ij},e_{kl})$ be an ordered pair of edges in a convexly
  punctured disc $\disc_n$ and let $B = \{i,j,k,l\}$.  If we restrict
  our attention to the convex subdisc $\disc_B$ (where $\disc_B$ is an
  $\epsilon$-neighborhood of the convex hull of the vertices indexed
  by the elements in $B$) then there are exactly five distinct
  possible configurations.  We call the possibilities crossing,
  noncrossing, identical, clockwise and counterclockwise.  When all
  four endpoints are distinct (i.e. when $|B| = 4$) these edges are
  either \emph{crossing} or \emph{noncrossing} depending on whether or
  not they intersect.  At the other extreme, when $e_{ij}$ and
  $e_{kl}$ have both endpoints in common, they are \emph{identical}.
  Finally, when these edges have exactly one endpoint in common, the
  convex hull of the three endpoints is a triangle and the edges occur
  as consecutive edges in its boundary cycle.  We call this
  arrangement \emph{clockwise} or \emph{counterclockwise} depending on
  the orientation of the boundary which ensures that $e_{kl}$ is the
  edge that occurs immediately after $e_{ij}$.  More colloquially we
  say that $e_{kl}$ is \emph{to the right (left)} of $e_{ij}$ and
  $e_{ij}$ is \emph{to the left (right)} of $e_{kl}$ when the ordered
  pair $(e_{ij},e_{kl})$ is clockwise (counterclockwise).  Examples
  are shown in Figure~\ref{fig:left-right}.
\end{defn}

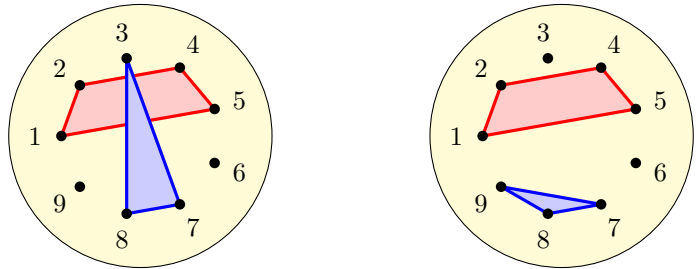
\begin{figure}
  \begin{tikzpicture}[scale=.7]
    \begin{scope}[xshift = -4cm]
      \filldraw[YellowPoly] (0,0) circle (2.5cm);
      \foreach \x in {1,2,...,9} {\coordinate (\x) at (220-40*\x:1.5cm);}
      \draw[RedPoly] (1)--(2)--(4)--(5)--cycle;
      \draw[BluePoly] (3)--(7)--(8)--cycle;
      \foreach \x in {1,2,...,9} {\fill (\x) circle (1mm);}
      \foreach \x in {1,2,...,9} {\draw (220-40*\x:2cm) node {\small \x};}
    \end{scope}

    \begin{scope}[xshift = 4cm]
      \filldraw[YellowPoly] (0,0) circle (2.5cm);
      \foreach \x in {1,2,...,9} {\coordinate (\x) at (220-40*\x:1.5cm);}
      \draw[RedPoly] (1)--(2)--(4)--(5)--cycle;
      \draw[BluePoly] (7)--(8)--(9)--cycle;
      \foreach \x in {1,2,...,9} {\fill (\x) circle (1mm);}
      \foreach \x in {1,2,...,9} {\draw (220-40*\x:2cm) node {\small \x};}
    \end{scope}
  \end{tikzpicture}
  \caption{The subsets $\{1,2,4,5\}$ and $\{3,7,8\}$ are crossing.
    The subsets $\{1,2,4,5\}$ and $\{7,8,9\}$ are
    noncrossing.\label{fig:noncrossing}}
\end{figure}

Noncrossing partitions are defined in a convexly punctured disc.

\begin{defn}[Noncrossing partitions]\label{def:noncrossing} 
  Let $\disc_n$ be a convexly punctured disc.  We say that two subsets
  $B, B' \subset [n]$ are \emph{noncrossing} when the convex hulls of
  the corresponding sets of vertices in $\disc_n$ are completely
  disjoint.  See Figure~\ref{fig:noncrossing}.  More generally, a
  partition $\sigma$ of the set $[n]$ is \emph{noncrossing} when its
  blocks are pairwise noncrossing.  Noncrossing partitions are usually
  ordered by refinement, so that $\sigma < \tau$ if and only if each
  block of $\sigma$ is contained in some block of $\tau$.  Under this
  ordering, the set of all noncrossing partitions form a bounded
  graded lattice denoted $NC_n$.  The poset $NC_4$ is shown in
  Figure~\ref{fig:nc4}.  The number of noncrossing partitions in
  $NC_n$ is given by the $n$-th Catalan number $C_n =
  \frac{1}{n+1}\binom{2n}{n}$.
\end{defn}

\begin{figure}
  \begin{tikzpicture}[scale = .5]
    \node (A) at (0,-4.5) [Rect] {}; \node (B) at (10,-2) [Rect] {};
    \node (C) at (6,-2) [Rect] {};   \node (D) at (2,-2) [Rect] {}; 
    \node (E) at (-2,-2) [Rect] {};  \node (F) at (-6,-2) [Rect] {};
    \node (G) at (-10,-2) [Rect] {}; \node (H) at (10,2) [Rect] {};
    \node (I) at (6,2) [Rect] {};    \node (J) at (2,2) [Rect] {};
    \node (K) at (-2,2) [Rect] {};   \node (L) at (-6,2) [Rect] {};
    \node (M) at (-10,2) [Rect] {};  \node (N) at (0,4.5) [Rect] {};
    
    \newcommand{\placedots}{\foreach \x in {1,2,3,4} {\coordinate (\x) at (315-90*\x:.7);}}
    \newcommand{\drawdots}{\foreach \x in {1,2,3,4} {\draw (\x) node [smallDot] {};}}
    \begin{scope}[BluePoly]
      \begin{scope}[shift={(0,-4.5)}] \placedots \drawdots \end{scope}
      \begin{scope}[shift={(10,-2)}]\placedots \draw (3)--(4); \drawdots \end{scope}
      \begin{scope}[shift={(6,-2)}]\placedots \draw (1)--(4);\drawdots \end{scope}
      \begin{scope}[shift={(2,-2)}] \placedots \draw (2)--(4);\drawdots \end{scope}
      \begin{scope}[shift={(-2,-2)}] \placedots \draw (1)--(3);\drawdots \end{scope}
      \begin{scope}[shift={(-6,-2)}] \placedots \draw (2)--(3);\drawdots \end{scope}
      \begin{scope}[shift={(-10,-2)}] \placedots \draw (1)--(2);\drawdots \end{scope}
      \begin{scope}[shift={(10,2)}] \placedots \filldraw (1)--(3)--(4)--cycle;\drawdots \end{scope}
      \begin{scope}[shift ={(6,2)}] \placedots \filldraw (2)--(3)--(4)--cycle;\drawdots \end{scope}	
      \begin{scope}[shift ={(2,2)}] \placedots \draw (1)--(2); \draw (3)--(4);\drawdots \end{scope}
      \begin{scope}[shift ={(-2,2)}] \placedots \draw (2)--(3); \draw (1)--(4);\drawdots \end{scope}
      \begin{scope}[shift ={(-6,2)}] \placedots \filldraw (1)--(2)--(4)--cycle;\drawdots \end{scope}
      \begin{scope}[shift={(-10,2)}] \placedots \filldraw (1)--(2)--(3)--cycle;\drawdots \end{scope}
      \begin{scope}[shift={(0,4.5)}] \placedots \filldraw (1)--(2)--(3)--(4)--cycle;\drawdots \end{scope}
    \end{scope}
    
    \path (A) edge (B) edge (C) edge (D) edge (E) edge (F) edge (G);
    \path (B) edge (H) edge (I) edge (J); 
    \path (C) edge (H) edge (K) edge (L);
    \path (D) edge (I) edge (L);
    \path (E) edge (H) edge (M);
    \path (F) edge (I) edge (K) edge (M);
    \path (G) edge (J) edge (L) edge (M);
    \path (N) edge (H) edge (I) edge (J) edge (K) edge (L) edge (M);
  \end{tikzpicture}
  \caption{Noncrossing partition lattice $NC_4$.\label{fig:nc4}}
\end{figure}
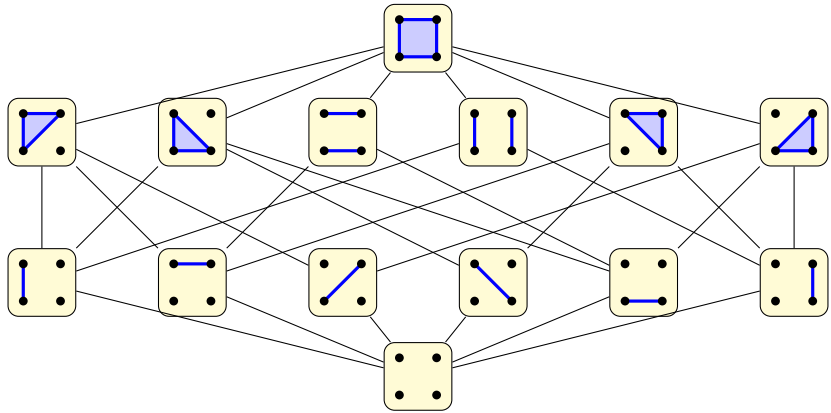

To each noncrossing partition there is a corresponding permutation.

\begin{defn}[Noncrossing permutations]\label{def:perms}
  To every subset $B \subset [n]$ with at least two elements we
  associate a permutation in $\sym_n$ obtained by linearly ordering
  the elements in $B$.  More generally, we associate a permutation to
  each noncrossing paritition by multiplying the permutations
  associated to each block and we call the result a \emph{noncrossing
    permutation}.  Since distinct blocks are disjoint, their
  permutations commute and the product is well-defined.  Thus $B =
  \{1,3,4\}$ becomes the permutation $(1,3,4)$ and the partition
  $\sigma = \{\{1,3,4\},\{2\},\{5,6,7,8,9\}\}$ becomes the permutation
  $(1,3,4)(5,6,7,8,9)$.  We identify each noncrossing partition with
  its corresponding noncrossing permutation, using the same symbol for
  both.  The permutation associated to the full set $[n]$ is an
  important $n$-cycle that we call $\delta$.
\end{defn}

As is well-known, the elements of the braid group can be identified
with (equivalence classes of) motions of $n$ distinct labeled points
in a disc such as $\disc_n$.  The dual simple braids are a finite set
of braids indexed by the noncrossing permutations as follows.

\begin{defn}[Rotations]\label{defn:rotate} 
  The \emph{dual Garside element} $s_\delta$ of the $n$-string braid
  group is the motion where each labeled point in $\disc_n$ moves
  clockwise along the boundary of the convex hull of all $n$ points to
  the next vertex.  More generally, for each set $B \subset \{1,
  \ldots, n \}$, let $P_B$ be the convex hull of the vertices indexed
  by $B$ and let $\disc_B$ be an $\epsilon$-neighborhood of $P_B$.
  The braid group element $s_B$ is a similiar motion restricted to the
  subdisc $\disc_B$, i.e. the vertices in the subdisc move clockwise
  along one side of the polygon $P_B$ to the next vertex, leaving all
  other vertices fixed.  See Figure~\ref{fig:rotation}.  When $B$ has
  at most $1$ element, the motion is trivial.  When $B$ has two
  elements, the points avoid collisions by passing on the left.
\end{defn} 

\begin{figure}
  \begin{tikzpicture}[scale=.7]
    \begin{scope}[xshift = -4cm]
      \filldraw[YellowPoly] (0,0) circle (2.5cm);
      \foreach \x in {1,2,...,9} {\coordinate (\x) at (220-40*\x:1.5cm);}
      \draw[very thick,color=blue,fill=blue!15] (1)--(3)--(7)--cycle;
      \drawArrow{(1)}{(3)}
      \drawArrow{(3)}{(7)}
      \drawArrow{(7)}{(1)}
      \foreach \x in {1,2,...,9} {\fill (\x) circle (1mm);}
      \foreach \x in {1,2,...,9} {\draw (220-40*\x:2cm) node {\small \x};}
    \end{scope}

    \draw[|-angle 90,very thick] (-1,0) -- node [above] {$s_{137}$} (1,0);
    
    \begin{scope}[xshift = 4cm]
      \filldraw[YellowPoly] (0,0) circle (2.5cm);
      \foreach \x in {1,2,...,9} {\coordinate (\x) at (220-40*\x:1.5cm);}
      \draw[very thick,color=blue,fill=blue!15] (1)--(3)--(7)--cycle;
      \foreach \x in {1,2,...,9} {\fill (\x) circle (1mm);}
      \foreach \x in {2,4,5,6,8,9} {\draw (220-40*\x:2cm) node {\small \x};}
      \draw (220-40*1:2cm) node {\small 7};
      \draw (220-40*3:2cm) node {\small 1};
      \draw (220-40*7:2cm) node {\small 3};
    \end{scope}
  \end{tikzpicture}
  \caption{The rotation $s_{137}$.\label{fig:rotation}}
\end{figure}
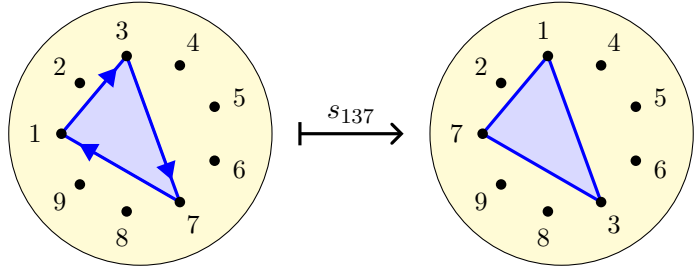

Rotations can be used to assign a braid to each noncrossing partition.

\begin{defn}[Dual simple braids]\label{def:simple-braids} 
  The \emph{dual simple braids} are elements of the braid group in
  one-to-one correspondence with the set of noncrossing partitions
  $NC_n$.  More precisely, for each noncrossing partition $\sigma$, we
  associate the product of the rotations corresponding to each of its
  blocks and call the result $s_\sigma$.  Because rotations of
  noncrossing blocks take place in disjoint subdiscs they commute and
  the resulting element in the braid group is well-defined.  Note that
  the noncrossing permutation $\sigma$ is the permutation of the
  vertices induced by $s_\sigma$.  The dual simple braids in
  $\braid_4$ written as products of rotations are shown in
  Figure~\ref{fig:b4-interval}.
\end{defn}

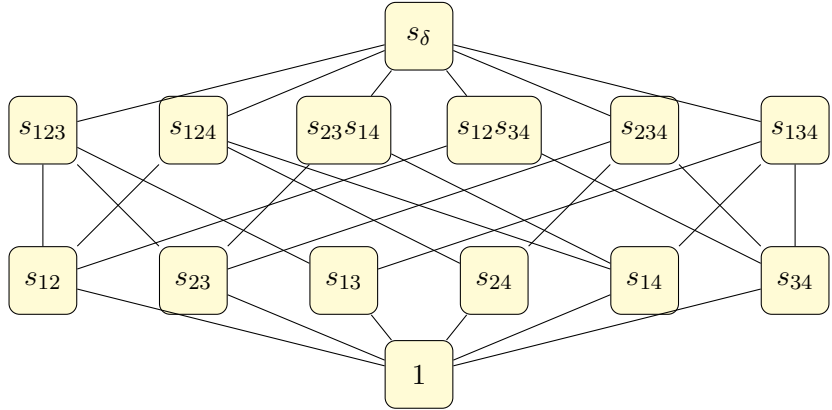
\begin{figure}
  \begin{tikzpicture}[scale = .5]
    \node (A) at (0,-4.5) [Rect] {$\id$}; 
    \node (B) at (10,-2) [Rect] {$s_{34}$};
    \node (C) at (6,-2) [Rect] {$s_{14}$};      
    \node (D) at (2,-2) [Rect] {$s_{24}$}; 
    \node (E) at (-2,-2) [Rect] {$s_{13}$};      
    \node (F) at (-6,-2) [Rect] {$s_{23}$};
    \node (G) at (-10,-2) [Rect] {$s_{12}$};
    \node (H) at (10,2) [Rect] {$s_{134}$};
    \node (I) at (6,2) [Rect] {$s_{234}$};     
    \node (J) at (2,2) [Rect] {$s_{12}s_{34}$};
    \node (K) at (-2,2) [Rect] {$s_{23}s_{14}$};  
    \node (L) at (-6,2) [Rect] {$s_{124}$};
    \node (M) at (-10,2) [Rect] {$s_{123}$};
    \node (N) at (0,4.5) [Rect] {$s_\delta$};

    \path (A) edge (B) edge (C) edge (D) edge (E) edge (F) edge (G);
    \path (B) edge (H) edge (I) edge (J); 
    \path (C) edge (H) edge (K) edge (L);
    \path (D) edge (I) edge (L);          
    \path (E) edge (H) edge (M);
    \path (F) edge (I) edge (K) edge (M); 
    \path (G) edge (J) edge (L) edge (M);
    \path (N) edge (H) edge (I) edge (J) edge (K) edge (L) edge (M);
  \end{tikzpicture}
  \caption{The dual simple elements in
    $\braid_4$.\label{fig:b4-interval}}
\end{figure}

It is useful to have specific names for four sets of dual simple
braids.

\begin{defn}[Four sets of simple braids]\label{def:four-sets}
  The \emph{standard generating set} for the braid group $\braid_n$
  consists of the $n-1$ rotations of the form $s_{ij}$ with $1 \leq i
  <n$ and $j=i+1$.  The \emph{dual generators} of $\braid_n$ are the
  set of all $N=\binom{n}{2}$ rotations $s_B$ where $B$ has exactly
  two elements.  The \emph{rotations} $s_B$ with $|B| \neq 1$ form a
  third set, and the full set of all dual simple braids form a fourth
  set.  We write $\std_n \subset \gen_n \subset \rot_n \subset
  \simp_n$ for these four nested sets, whose sizes are $n-1$, $N$,
  $2^n-n$ and $C_n$ (the $n$-th Catalan number).  When $n=4$, these
  sets have $3$, $6$, $11$ and $14$ elements.
\end{defn}

The multiplication in the symmetric group can be used to extract other
information about noncrossing partitions.

\begin{defn}[Multiplication]\label{def:mult}
  For consistency with our latter conventions we view permutations as
  functions and thus we multiply them from \emph{right to left}.  For
  example, the product $(1,2,3) \cdot (3,4,5)$ is $(1,2,3,4,5)$.  More
  generally, if $B$, $\{i\}$ and $C$ are pairwise disjoint subsets of
  $[n]$ such that there is a place to start reading the boundary cycle
  of the convex hull of the points corresponding to the elements in $B
  \cup \{i\} \cup C$ so that, reading clockwise, one encounters all of
  the vertices indexed by elements in $B$, followed by $p_i$, followed
  by all of the vertices indexed by the elements in $C$, then $s_{Bi}
  s_{iC} = s_{BiC}$.  Here we follow the conventions of \cite{MaMc09},
  removing parentheses from singletons, and using juxtaposition to
  indicate union.
\end{defn}

The reader should be careful to note that our multiplication
convention differs from much of the literature on the braid groups
where multplication is from left to right.  Thus extra vigilance is
required.  The general multiplication rule given above, for example,
is stated in a slightly different form in \cite{MaMc09}.  The
multiplication can be used to define left and right complements of
noncrossing permutations.

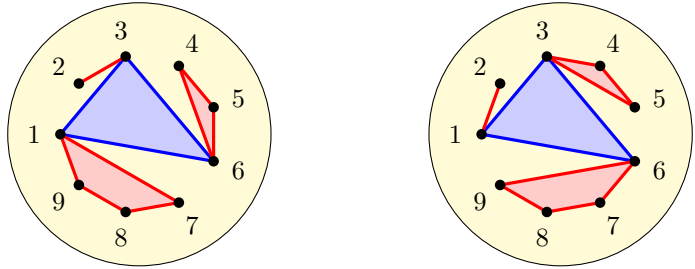
\begin{figure}
  \begin{tikzpicture}[scale=.7]
    \begin{scope}[xshift=-4cm]
      \filldraw[YellowPoly] (0,0) circle (2.5cm);
      \foreach \x in {1,2,...,9} {\coordinate (\x) at (220-40*\x:1.5cm);}
      \draw[BluePoly] (1)--(3)--(6)--cycle;
      \draw[RedPoly] (2)--(3);
      \draw[RedPoly] (6)--(5)--(4)--cycle;
      \draw[RedPoly] (1)--(7)--(8)--(9)--cycle;
      \foreach \x in {1,2,...,9} {\fill (\x) circle (1mm);}
      \foreach \x in {1,2,...,9} {\draw (220-40*\x:2cm) node {\small \x};}
    \end{scope}
    \begin{scope}[xshift=4cm]
      \filldraw[YellowPoly] (0,0) circle (2.5cm);
      \foreach \x in {1,2,...,9} {\coordinate (\x) at (220-40*\x:1.5cm);}
      \draw[BluePoly] (1)--(3)--(6)--cycle;
      \draw[RedPoly] (1)--(2);
      \draw[RedPoly] (3)--(4)--(5)--cycle;
      \draw[RedPoly] (6)--(7)--(8)--(9)--cycle;
      \foreach \x in {1,2,...,9} {\fill (\x) circle (1mm);}
      \foreach \x in {1,2,...,9} {\draw (220-40*\x:2cm) node {\small
          \x};}
    \end{scope}
  \end{tikzpicture}
  \caption{$(23)(456)(1789)$ is the left complement of $(136)$ in
    $\sym_9$ and $(12)(345)(6789)$ is its right
    complement.\label{fig:complements}}
\end{figure}

\begin{defn}[Complements]\label{def:complements}
  Let $\sigma$ be a noncrossing permutation in $\sym_n$ and recall
  that $\delta$ is the $n$-cycle $(1,2,\ldots, n)$.  The \emph{left
    complement} of $\sigma$ is the unique element $\sigma'$ such that
  $\sigma' \sigma = \delta$ and its \emph{right complement} is the
  unique element $\sigma''$ such that $\sigma \sigma'' = \delta$.  We
  denote these permutations by $\sigma'=\lc(\sigma)$ and $\sigma'' =
  \rc(\sigma)$.  The permutation $\lc(\sigma)$ is always also a
  noncrossing permutation and, in fact, the edges in its blocks are
  precisely those that are to the left or noncrossing with respect to
  each of the edges in the blocks of $\sigma$.  Similarly
  $\rc(\sigma)$ is a noncrossing permutation whose blocks are formed
  by the edges that are to the right or noncrossing with respect to
  each edge in a block of $\sigma$.  See Figure~\ref{fig:complements}.
\end{defn}

The following observation is not crucial to our results, but we
sometimes use this language.

\begin{rem}[Hypertrees]\label{rem:hypertrees}
  A \emph{hypergraph} is a generalization of a graph where its
  \emph{hyperedges} are allowed to span more than two vertices, and a
  \emph{hypertree} is the natural generalization of a tree.  As can be
  seen in Figure~\ref{fig:complements}, the blocks of the noncrossing
  partition associated to a dual simple element and the blocks of one
  of its complements together form the hyperedges of a planar
  hypertree.
\end{rem}

\begin{figure}
  \begin{tikzpicture}[scale=.7]
      \filldraw[YellowPoly] (0,0) circle (2.5cm);
      \foreach \x in {1,2,...,9} {\coordinate (\x) at (220-40*\x:1.5cm);}
      \draw[RedPoly] (1)--(2)--(8)--(9)--cycle;
      \draw[RedPoly] (1)--(7)--(8)--(9)--cycle;
      \draw[RedPoly] (1)--(5)--(8)--(9)--cycle;
      \draw[RedLine,dashed] (1)--(5)--(8)--(9)--cycle;
      \draw[RedLine,dashed] (1)--(2)--(8);
      \draw[RedLine,dashed] (1)--(7)--(8);
      \draw[BluePoly] (2)--(3)--(4)--(5)--cycle;
      \draw[BluePoly] (5)--(6)--(7)--cycle;
      \foreach \x in {1,2,...,9} {\fill (\x) circle (1mm);}
      \foreach \x in {1,2,...,9} {\draw (220-40*\x:2cm) node {\small \x};}
  \end{tikzpicture}
  \caption{If $\sigma_1 = (2,3,4,5)$ and $\sigma_2 =(5,6,7)$, then the
    permutations $\sigma_3$, $\sigma_4$ and $\sigma_5$ defined in
    Definition~\ref{def:five-perms} are $\sigma_3 = (1,7,8,9)$,
    $\sigma_4 = (1,5,8,9)$ and $\sigma_5 =
    (1,2,8,9)$.\label{fig:five-perms}}
\end{figure}
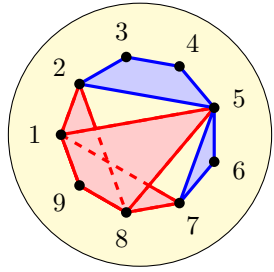

The Hasse diagram of the noncrossing partition lattice can actually be
viewed as a portion of the Cayley graph of $\sym_n$ with respect the
generating set of all transpostions, specifically the portion between
the identity and the element $\delta$.  This way of looking at the
noncrossing partition lattice combined with the fact that the set of
transpositions in $\sym_n$ is closed under conjugation, helps to
explain why factorizations of $\delta$ as a product of noncrossing
permutations are so flexible.  The only aspect of this flexibility
that we need here is the following.

\begin{defn}[Five permutations]\label{def:five-perms}
  If $\sigma_1$ and $\sigma_2$ are permutations in $\sym_n$ such that
  $\sigma_1$, $\sigma_2$ and their product $\sigma_1 \sigma_2$ are all
  three noncrossing, then there exist noncrossing permutations
  $\sigma_3$, $\sigma_4$ and $\sigma_5$ such that $\delta = \sigma_1
  \sigma_2 \sigma_3 = \sigma_1 \sigma_4 \sigma_2 = \sigma_5 \sigma_1
  \sigma_2$.  The permutations $\sigma_5$ and $\sigma_3$ are simply
  the left and right complements of the product $\sigma_1 \sigma_2$,
  while $\sigma_4$ is obtained by conjugation.  An example is shown in
  Figure~\ref{fig:five-perms}.
\end{defn}

\section{Simplices}\label{sec:simplices}
In this section we discuss the geometry of euclidean simplices with
$n$ labeled vertices and, in particular, how this geometry changes
under certain carefully controlled deformations.  As in
\cite{BrMc-factor}, we start by distinguishing between points and
vectors.

\begin{defn}[Points]\label{def:pts}
  Let $V$ be an $(n-1)$-dimensional \emph{real vector space} with a
  fixed positive definite inner product but no fixed basis and let $E$
  be an $(n-1)$-dimensional \emph{euclidean space}, which may be
  defined as a set with a fixed simply-transitive action of the
  additive group of $V$.  The structure of $E$ is essentially that of
  $V$ but the location of the origin has been forgotten.  The elements
  of $V$ are \emph{vectors}, the elements of $E$ are \emph{points},
  and we write $\langle u, v \rangle$ for the inner product of vectors
  $u$ and $v$.  Two points $p$ and $p'$ determine a line segment $e$
  called an \emph{edge} and $p$ and $p'$ are its \emph{endpoints}.  By
  the simply-transitive action of $V$ on $E$, they also determine two
  vectors: the unique vector $v$ that sends $p$ to $p'$ and the vector
  $-v$ which sends $p'$ to $p$.  The pair $\pm v$ is a \emph{lax
    vector}.  When the points involved are labeled, say $p_i$ and
  $p_j$, we write $e_{ij} = e_{ji}$ for the edge they span, and $v =
  v_{ij}$ and $-v = v_{ji}$ for the two vectors they determine.  The
  \emph{norm} of a vector $v$ is $\langle v, v \rangle$, which is also
  the square of the length of the corresponding edge $e$.  We write
  $\norm:V \to \R$ for the norm map.  Note that the norm of a lax
  vector is well-defined since $\norm(v) = \norm(-v)$ and the norm of
  an edge is the norm of the lax vector determined by its endpoints.
  When the points involved are labeled we write $a_{ij}$ for
  $\norm(v_{ij})$.
\end{defn}

\begin{defn}[Simplices]\label{def:simplices}
  A set $\{p_i\}$ of $n$ labeled points in $E$ is in \emph{general
    position} if this set is not contained in any proper affine
  subspace of $E$, and the convex hull of such a set is a
  \emph{labeled euclidean simplex} $\Delta$ of dimension~$(n-1)$.  For
  any such labeled euclidean simplex $\Delta$ we use subsets of
  punctures in the convexly punctured disc $\disc_n$ to describe
  various simplicial faces of $\Delta$ via their vertex labelings.
  For example, the three blocks of the left complement of $s_{136}$
  shown in Figure~\ref{fig:complements} correspond to an edge, a
  triangle and a tetrahedron in any $8$-dimensional simplex $\Delta$
  with $9$ labeled vertices.
\end{defn}

We are primarily interested in the isometry class of a labeled
euclidean simplex $\Delta$ and this is completely determined by the
ordered list of the norms of its edges.

\begin{defn}[Edge norm vectors]\label{def:tuples}
  Let $\Delta$ be a labeled euclidean simplex with $n$ vertices.  The
  \emph{edge norm vector} of $\Delta$ is a column vector $\vv$ of the
  $N=\binom{n}{2}$ positive real numbers $a_{ij}$ which are the norms
  of its edges $e_{ij}$, listed in the standard lexicographic order of
  the edges as discussed in Definition~\ref{def:edges}.
\end{defn}

Edge norm vectors characterize isometry classes of labeled euclidean
simplices and as a result when we reshape a labeled euclidean simplex,
these changes to its geometry are captured by the modifications that
occur in its edge norm vector.  The well-known formula $2 \langle u,v
\rangle = \norm(u+v) - \norm(u) -\norm(v)$ shows that inner products
of vectors can be calculated in terms of their norms, but we need a
slightly more general formula that computes the inner product of two
vectors determined by four possibly distinct points in a euclidean
space~$E$.

\begin{figure}
  \begin{tikzpicture}[scale = .8]
    \coordinate (1) at (0,0);
    \coordinate (2) at (3.5,5.46);
    \coordinate (3) at (7,1);
    \coordinate (4) at (5.4,-1.2);
    \draw[GreenPoly] (1)--(2)--(3)--(4)--cycle;
    \draw (2)--(3)--(4)--cycle;
    \draw (1)--(2) node [midway,above left,black] {$a$};
    \draw[dashed] (1)--(3) node [midway,above,black] {$b$};
    \draw[black] (1)--(4) node [midway,below,black] {$c$}; 
    \draw (2)--(3) node [midway,above right,black] {$d$};
    \draw (2)--(4) node [midway,right,black] {$e$};
    \draw (3)--(4) node [midway,right,black] {$f$};
    \fill (1) circle (.6mm) node[anchor=east] {$p_i$};
    \fill (2) circle (.6mm) node[anchor=south] {$p_j$};
    \fill (3) circle (.6mm) node[anchor=west] {$p_k$};
    \fill (4) circle (.6mm) node[anchor=north] {$p_l$};
  \end{tikzpicture}
  \caption{Tetrahedron determined by $4$ points, edges labeled by
    norm.}\label{fig:tetrahedron}
\end{figure}
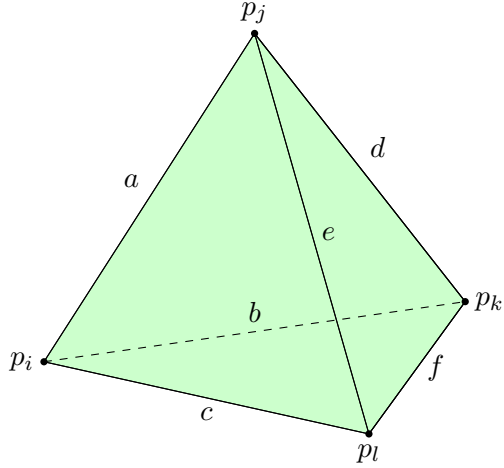

\begin{prop}[Inner products and norms]\label{prop:inner-prods} 
  If $p_i, p_j, p_k$ and $p_l$ are four not necessarily distinct
  points in a euclidean space $E$, then the inner product $2\langle
  v_{ij}, v_{kl} \rangle = a_{il} + a_{jk} - a_{ik} - a_{jl}$.
\end{prop}

\begin{proof}
  To improve readability, we write $a$, $b$, $c$, $d$, $e$ and $f$ for
  the norms $a_{ij}$, $a_{ik}$, $a_{il}$, $a_{jk}$, $a_{jl}$ and
  $a_{kl}$, respectively.  The situation under discussion is shown in
  Figure~\ref{fig:tetrahedron}.  Expanding the norms of $v_{ik} =
  v_{ij} + v_{jk}$ and $v_{jl} = v_{jk} + v_{kl}$ produces the
  identities $b = a + d + 2\langle v_{ij}, v_{jk} \rangle$ and $e = d
  + f + 2\langle v_{jk}, v_{kl} \rangle$. Expanding the norm of
  $v_{il} = v_{ij} + v_{jk} + v_{kl}$ produces $c = a + d + f +
  2\langle v_{ij}, v_{jk} \rangle + 2 \langle v_{jk}, v_{kl} \rangle +
  2 \langle v_{ij}, v_{kl}\rangle$. Thus $2 \langle v_{ij}, v_{jk}
  \rangle = b-a-d$, $2\langle v_{jk}, v_{kl} \rangle = e - d - f$, and
  $2 \langle v_{ij}, v_{kl} \rangle = c - a - d - f - (b - a - d) - (e
  - d - f) = c + d -b -e$.
\end{proof}

The following definition identifies a class of geometric reshapings of
labeled euclidean simplices that are particularly elegant and easy to
describe. We call them edge rescalings and, as far as we are aware,
they have not been previously discussed in the literature.

\begin{defn}[Edge rescaling]\label{def:edge-rescaling}
  Let $\Delta$ and $\Delta'$ be two labeled euclidean simplices with
  $n$ vertices situated in a common euclidean space.  We say that an
  edge $e_{ij}$ in $\Delta$ is merely \emph{rescaled} if it and the
  corresponding edge $e'_{ij}$ in $\Delta'$ point in the same
  direction.  More generally, we say that $\Delta'$ is an \emph{edge
    rescaling} of $\Delta$ if there exist enough pairs of
  corresponding edges pointing in the same direction (but with
  possibly different lengths) to form a basis for the vector space out
  of these common direction vectors.
\end{defn}

\begin{rem}[Spanning trees]\label{rem:trees}
  Let $\Delta'$ be a labeled euclidean simplex which is an edge
  rescaling of $\Delta$.  By definition there are sufficiently many
  edges that are merely rescaled to form a basis out of the
  corresponding vectors and a minimal set of rescaled edges in the
  $1$-skeleton of $\Delta$ would form a spanning tree in this complete
  graph.  There might, however, be more than one such spanning tree of
  merely rescaled edges.  When two edges in $\Delta$ share a common
  endpoint and are both rescaled by the same scale factor, the
  triangle they span, and in particular the third edge in that
  triangle, is also rescaled by the same scale factor.  Thus any two
  of these three edges could be included in the spanning tree.  In
  fact, it would be more canonical to identify the \emph{maximal}
  simplicial faces that are merely rescaled.  For each scale factor,
  there would be a partition of the vertices into maximal subsimplices
  rescaled by that factor and the blocks in all of these partitions
  together would form a spanning hypertree in the sense of
  Remark~\ref{rem:hypertrees}.  The full variety of spanning trees
  which satisfy the edge rescaling definition are selected from the
  edges inside the blocks of such a canonical spanning hypertree.
\end{rem}

\begin{defn}[Edge rescaling maps]\label{def:edge-rescaling-maps}
  One thing to note is that for every spanning tree $T$ in the
  $1$-skeleton of a labeled euclidean simplex $\Delta$ and for every
  set of positive real scale factors for these edges, there does exist
  a rescaled simplex $\Delta'$ in which these edges are rescaled by
  these factors since the rescaled tree $T'$ formed by assembling the
  rescaled edges as before tells us how the vertices should be
  arranged.  In particular, the rescaling of $\Delta$ only depends on
  the set of merely rescaled edges and the scale factors used to
  rescale them.  Thus there is a well-defined \emph{edge rescaling
    map} $R$ from the space of labeled euclidean simplices to itself
  which, based only on such data, rescales each $\Delta$ in the set to
  a new simplex $\Delta'$.  Such a map $R$ is clearly invertible since
  rescaling the same edges by the multiplicative inverse of each scale
  factor returns each $\Delta'$ to $\Delta$.  The edge rescaling maps
  we are primarily interested in are those where every scale factor is
  $1$ or $q$.  We call these \emph{$q$-rescalings} and we introduce a
  special notation for them.  Let $R = R^\sigma_\tau$ denote the
  $q$-rescaling where the blocks of the partition $\sigma$ index the
  subsimplices rescaled by $q$ and the blocks of the partition $\tau$
  index the subsimplices which are fixed, i.e. rescaled by a factor of
  $1$.
\end{defn}

The next proposition shows that the effect of $R$ is encoded in a
matrix.

\begin{prop}[Edge rescaling matrices]\label{prop:edge-rescaling-matrices}
  The effect of an edge rescaling map $R$ on an edge norm vector $\vv$
  for a labeled euclidean simplex $\Delta$ is captured by an $N$ by
  $N$ matrix $M$ whose entries only depend on the map $R$ and not on
  the vector $\vv$ or the simplex $\Delta$.  In particular, $R(\vv) =
  M \cdot \vv$ for all $\vv$ and~$\Delta$.
\end{prop}

\begin{proof}
  Let $T$ be a spanning set of edges that are rescaled by $R$. For
  each edge $e_{ij}$ we can use paths in the spanning tree $T$ to find
  a linear combination of vectors associated with rescaled edges whose
  sum is $v_{ij}$.  The new vector $v'_{ij}$, by definition, is a
  similar sum where the vectors in the sum are rescaled according to
  the scale factors of $R$.  Thus the norm of $v'_{ij}$ can be
  expanded as a linear combination of inner products of vectors whose
  edges belong to the tree $T$ and the coefficients of this linear
  combination are independent of the original edge norms.  Next, by
  Proposition~\ref{prop:inner-prods}, the inner products can be
  rewritten as linear combinations of the original edge norms.
  Substituting these in produces a formula for each new edge norm
  $a'_{ij}$ as a linear combination of the old edge norms in $\vv$
  with coefficients that are independent of $\vv$.  The matrix $M$ is
  formed by collecting these coefficients.
\end{proof}

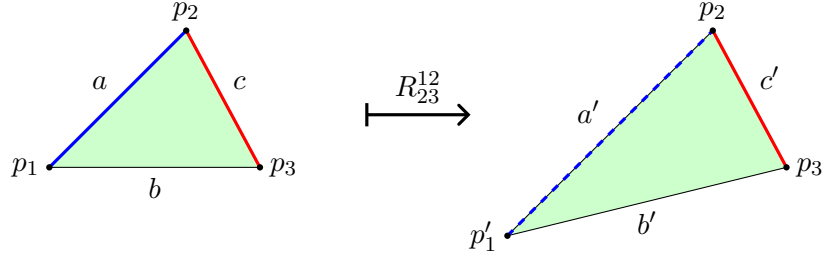
\begin{figure}
  \begin{tikzpicture}[join=round,scale=.7]
    \begin{scope}[xshift=-5cm]
      \coordinate (pa) at (0,0);
      \coordinate (pb) at (2.6,2.6);
      \coordinate (pc) at (4,0);
      \draw[GreenPoly] (pa)--(pb)--(pc)--cycle;
      \draw ($(pa)!.5!(pb)$) node [above left] {$a$};
      \draw ($(pa)!.5!(pc)$) node [below] {$b$};
      \draw ($(pb)!.5!(pc)$) node [above right] {$c$};
      \draw[BlueLine] (pa)--(pb);
      \draw[RedLine] (pb)--(pc);
      \fill (pa) circle (.6mm) node[anchor=east] {$p_1$};
      \fill (pb) circle (.6mm) node[anchor=south] {$p_2$};
      \fill (pc) circle (.6mm) node[anchor=west] {$p_3$};
    \end{scope}

    \begin{scope}[shift={(2cm,1cm)}]
      \draw[|-angle 90,very thick] (-1,0) -- node [above] {$R^{12}_{23}$} (1,0);
    \end{scope}

    \begin{scope}[xshift=5cm]
      \coordinate (pa) at (-1.3,-1.3);
      \coordinate (pb) at (2.6,2.6);
      \coordinate (pc) at (4,0);
      \draw[GreenPoly] (pa)--(pb)--(pc)--cycle;
      \draw ($(pa)!.5!(pb)$) node [above left] {$a'$};
      \draw ($(pa)!.5!(pc)$) node [below] {$b'$};
      \draw ($(pb)!.5!(pc)$) node [above right] {$c'$};
      \draw[BlueLine,dashed] (pa)--(pb);
      \draw[RedLine] (pb)--(pc);
      \fill (pa) circle (.6mm) node[anchor=east] {$p'_1$};
      \fill (pb) circle (.6mm) node[anchor=south] {$p_2$};
      \fill (pc) circle (.6mm) node[anchor=west] {$p_3$};
    \end{scope}
  \end{tikzpicture}
  \caption{An edge reshaping $R=R^{12}_{23}$ that rescales $e_{12}$ by
    a factor of $q$ and fixes $e_{23}$ (i.e. rescales $e_{23}$ by a
    factor of $1$).\label{fig:tri-rescale}}
\end{figure}

For simplicity we use the same symbol $R$ to denote both the edge
rescaling map and the edge rescaling matrix that was called $M$ in the
proposition.  As an explicit example of this process, consider the
$q$-rescaling of a triangle shown in Figure~\ref{fig:tri-rescale}.

\begin{prop}[Edge rescaling a triangle]\label{prop:triangle}  
  If $\Delta$ is a labeled euclidean triangle and $\Delta'$ is the
  labeled euclidean triangle obtained by the edge rescaling $R =
  R^{12}_{23}$, then the edge norms of $\Delta'$ can be computed from
  the edge norms of $\Delta$ as follows: $a_{12}' = q^2 a_{12}$,
  $a_{13}' = (q^2-q)a_{12} + q a_{13} + (1-q) a_{23}$, and $a_{23}' =
  a_{23}$.  In particular, $R$ has the effect of multiplying the edge
  norm vector of $\Delta$ by a matrix with entries in $\Z[q]$.
  \begin{equation}
    \vv' = 
    \begin{bmatrix} a'_{12} \\a'_{13}\\ a'_{23} \end{bmatrix}  
    =\begin{bmatrix}
    q^2 & 0 & 0\\
    q^2-q & q & 1-q\\
    0 & 0 & 1
    \end{bmatrix}
    \cdot   
    \begin{bmatrix} a_{12} \\a_{13}\\ a_{23} \end{bmatrix}  
    = R \cdot   
    \begin{bmatrix} a_{12} \\a_{13}\\ a_{23} \end{bmatrix}  
    = R \cdot \vv
  \end{equation}
\end{prop}

\begin{proof} 
  For clarity let $a$, $b$ and $c$ be the edge norms $a_{12}$,
  $a_{13}$ and $a_{23}$ in $\Delta$ and add primes for the
  corresponding edge norms in $\Delta'$.  In the original triangle we
  have $a=\langle v_{12},v_{12} \rangle$, $c=\langle v_{23},v_{23}
  \rangle$ and $b = \langle v_{13}, v_{13} \rangle = \langle
  v_{12}+v_{23},v_{12}+v_{23} \rangle = a + 2 \langle
  v_{12},v_{23}\rangle +c$.  Thus $2 \langle v_{12},v_{23} \rangle =
  b-a-c$.  In the new triangle we have $c' = \langle v_{23},v_{23}
  \rangle = c$, $a' = \langle q\cdot v_{12}, q\cdot v_{12} \rangle =
  q^2 a$, and $b' = \langle q\cdot v_{12} + v_{23}, q \cdot v_{12} +
  v_{23}\rangle = q^2 a + 2q \langle v_{12}, v_{23}\rangle + c = q^2a
  + q(b-a-c) +c = (q^2-q)a + qb + (1-q)c$ as required.
\end{proof}

Many of the properties of the matrix $R = R^{12}_{23}$ extend to all
$q$-rescalings.

\begin{prop}[Quadratic matrices]
  Let $R = R^\sigma_\tau$ be a $q$-rescaling of a labeled euclidean
  simplex $\Delta$ with $n$ vertices.  The effect of $R$ on the edge
  norm vector $\vv$ of $\Delta$ is to multiply from the left by an $N$
  by $N$ matrix with entries in $\Z[q]$ of degree at most $2$.
\end{prop}

\begin{proof}
  The proof is the same as that of
  Proposition~\ref{prop:edge-rescaling-matrices} but with the
  additional observation that the coefficients are at most quadratic
  polynomials in $q$ since there is at most one $q$ coming from each
  side of the inner product.
\end{proof}

\begin{figure}
  \begin{tikzpicture}[scale=.6]
    \begin{scope}[xshift=-7cm]
      \coordinate (1) at (0,0);
      \coordinate (2) at (3.5,5.46);
      \coordinate (3) at (7,1);
      \coordinate (4) at (5.4,-1.2);
      \draw[GreenPoly] (1)--(2)--(3)--(4)--cycle;
      \draw[RedPoly] (2)--(3)--(4)--cycle;
      \draw[BluePoly] (1)--(2) node [midway,above left,black] {$a$};
      \draw[dashed] (1)--(3) node [midway,above,black] {$b$};
      \draw[black] (1)--(4) node [midway,below,black] {$c$}; 
      \draw[RedPoly] (2)--(3) node [midway,above right,black] {$d$};
      \draw[RedPoly] (2)--(4) node [midway,right,black] {$e$};
      \draw[RedPoly] (3)--(4) node [midway,right,black] {$f$};
      \fill (1) circle (.6mm) node[anchor=east] {$p_1$};
      \fill (2) circle (.6mm) node[anchor=south] {$p_2$};
      \fill (3) circle (.6mm) node[anchor=west] {$p_3$};
      \fill (4) circle (.6mm) node[anchor=north] {$p_4$};
    \end{scope}

    \begin{scope}[shift={(3cm,1cm)}]
      \draw[|-angle 90,very thick] (-1,0) -- node [above] {$R^{12}_{234}$} (1,0);
    \end{scope}

    \begin{scope}[xshift=5cm]
      \coordinate (1) at (-.875,-1.365);
      \coordinate (2) at (3.5,5.46);
      \coordinate (3) at (7,1);
      \coordinate (4) at (5.4,-1.2);
      \draw[GreenPoly] (1)--(2)--(3)--(4)--cycle;
      \draw[RedPoly] (2)--(3)--(4)--cycle;
      \draw[dashed,BluePoly] (1)--(2) node [midway,above left,black] {$a'$};
      \draw[dashed] (1)--(3) node [midway,above,black] {$b'$};
      \draw[black] (1)--(4) node [midway,below,black] {$c'$}; 
      \draw[RedPoly] (2)--(3) node [midway,above right,black] {$d'$};
      \draw[RedPoly] (2)--(4) node [midway,right,black] {$e'$};
      \draw[RedPoly] (3)--(4) node [midway,right,black] {$f'$};
      \fill (1) circle (.6mm) node[anchor=east] {$p'_1$};
      \fill (2) circle (.6mm) node[anchor=south] {$p_2$};
      \fill (3) circle (.6mm) node[anchor=west] {$p_3$};
      \fill (4) circle (.6mm) node[anchor=north] {$p_4$};
    \end{scope}
  \end{tikzpicture}
  \caption{The edge rescaling $R^{12}_{234}$ which fixes the triangle
    $\Delta_{234}$ and rescales edge $e_{12}$ by a factor of
    $q$.\label{fig:tetra-rescale}}
\end{figure}
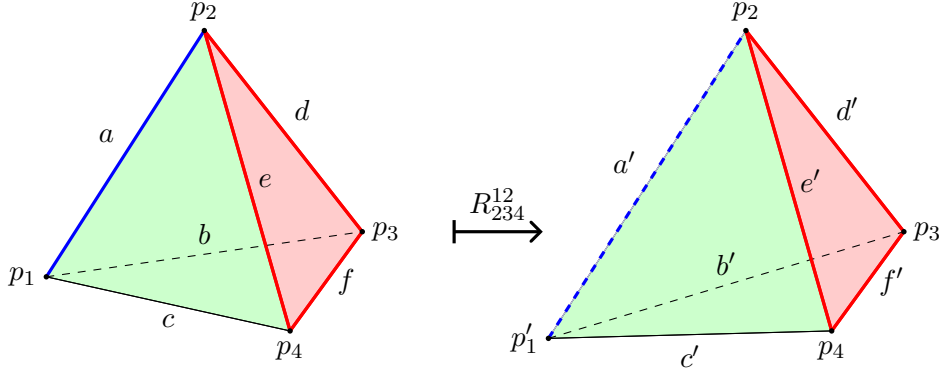

Our second example is very similar to the first.

\begin{exmp}[Rescaling a tetrahedron]\label{ex:tetra}
  Consider the edge rescaling $R^{12}_{234}$ as shown in
  Figure~\ref{fig:tetra-rescale} and note that we can compute all of
  the new edge norms using Proposition~\ref{prop:triangle}.  For
  clarity we write $a$ through $f$ for $a_{12}$ through $a_{34}$ in
  lexicographic order.  The effect of the map $R = R^{12}_{234}$ is as
  follows:
  \begin{equation}\label{eq:col}
    R \cdot 
    \begin{bmatrix} a\\ b\\ c\\ d\\ e\\ f \end{bmatrix} 
    = \begin{bmatrix} 
      q^2a \\ 
      (q^2-q)a + qb + (1-q)d \\ 
      (q^2-q)a + qc + (1-q)e \\ 
      d \\ 
      e \\ 
      f 
    \end{bmatrix}
    = \begin{bmatrix} a'\\ b'\\ c'\\ d'\\ e'\\ f' \end{bmatrix}
  \end{equation}
  Thus the matrix that encodes the rescaling $R$ is
  \begin{equation}
    R = \begin{bmatrix} 
      q^2 & 0 & 0 & 0 & 0 & 0 \\ 
      q^2-q & q & 0 & 1-q & 0 & 0 \\ 
      q^2-q & 0 & q & 0 & 1-q & 0 \\ 
      0 & 0 & 0 & 1 & 0 & 0 \\ 
      0 & 0 & 0 & 0 & 1 & 0 \\ 
      0 & 0 & 0 & 0 & 0 & 1 
    \end{bmatrix}
  \end{equation}
\end{exmp}

When the simplices are higher-dimensional, additional notation is
needed.

\begin{defn}[Row descriptions]
  Since the new edge norms are determined by the rows of the matrix $R
  = R^\sigma_\tau$, we introduce a way to describe these rows.  Recall
  that $e_{ij}$ denotes an edge in the punctured disc $\disc_n$ and an
  edge in a labeled euclidean simplex $\Delta$.  We now add a third
  interpretation: as an element of the canonical basis of the vector
  space $\R^N$ containing the edge norm vectors.  Using this
  interpretation of $e_{ij}$ as basis vectors, the second row of the
  matrix for $R=R^{12}_{23}$ as given in
  Proposition~\ref{prop:triangle} is the row vector $(q^2-q,q,1-q)$
  or, equivalently, it is the linear combination $(q^2-q) e_{12} + q
  e_{13} + (1-q) e_{23}$.  To select the second row one would multiply
  the matrix $R$ on the right by the row vector $(0,1,0)$ which is
  just the vector $e_{13}$.  In other words, $(e_{13}) R =(q^2-q)
  e_{12} + q e_{13} + (1-q) e_{23}$.
\end{defn}

\begin{rem}[Left and right]
  The linear combination that describes the image of a basis vector
  acted on by a rescaling matrix $R$ from the right looks very similar
  to the corresponding entry in the edge norm vector when acted on by
  $R$ from the left precisely because both are essentially encoding
  the entries in one row of $R$.  Nevertheless, this switch between
  left and right has the potential to be slightly confusing.
\end{rem}

It turns out that the linear combination that describes the $e_{kl}$
row of the matrix $R^{ij}_{\rc(ij)}$ (or $R^{ij}_{\lc(ij)}$) only
depends on the geometric relationship between the edges $e_{ij}$ and
$e_{kl}$ in the punctured disc $\disc_n$. Thus, in our row
descriptions of these matrices the final column uses the language of
Definition~\ref{def:edge-pairs} to describe how $e_{kl}$ is situated
relative to $e_{ij}$.

\begin{exmp}[Rescaling a boundary edge]\label{ex:boundary-edge}
  In this example we give explicit row descriptions for the
  $q$-rescalings which stretch a single boundary edge while fixing
  either its left or its right complement.  We start with the row
  description of the matrix $R = R^{12}_{\rc(12)}$.
  \begin{equation}
    (e_{kl})R = \left \{
    \begin{array}{cl}  
      q^2 e_{kl} & \textrm{identical ($k=1,l=2$)}\\
      e_{kl} & \textrm{noncrossing ($k,l >2$)}\\
      e_{kl} & \textrm{to the right ($k=2$)}\\
      (q^2-q) e_{12} + q e_{kl} + (1-q) e_{2l}  & \textrm{to the left ($k=1$)}\\
    \end{array}
    \right.
  \end{equation}
  This is a natural generalization of the trianglar and tetrahedral
  examples in the new notation.  The row description of $R =
  R^{ij}_{\rc(ij)}$ with $j=i+1 \mod n$ is only slightly more
  complicated.  We write $e_\third$ for the third edge of the triangle
  when $e_{ij}$ and $e_{kl}$ have exactly one endpoint in common.
  \begin{equation}
    (e_{kl})R = \left \{
    \begin{array}{cl}  
      q^2 e_{kl} & \textrm{identical}\\
      e_{kl} & \textrm{noncrossing}\\
      e_{kl} & \textrm{to the right}\\
      (q^2-q) e_{ij} + q e_{kl} + (1-q) e_\third  & \textrm{to the left}\\
    \end{array}
    \right.
  \end{equation}
  Switching from the right complement to the left complement causes
  only very minor changes.  The row description of the matrix $R =
  R^{12}_{\lc(12)}$ is as follows.
  \begin{equation}
    (e_{kl})R = \left \{
    \begin{array}{cl}  
      q^2 e_{kl} & \textrm{identical ($k=1,l=2$)}\\
      e_{kl} & \textrm{noncrossing ($k,l >2$)}\\
      e_{kl} & \textrm{to the left ($k=1$)}\\
      (q^2-q) e_{12} + q e_{kl} + (1-q) e_{1l}  & \textrm{to the right ($k=2$)}\\
    \end{array}
    \right.
  \end{equation}
  And finally, we list the row description of $R^{ij}_{\lc(ij)}$ with
  $j=i+1 \mod n$, with the same convention that $e_\third$ denotes the
  third edge of the triangle when $e_{ij}$ and $e_{kl}$ have exactly
  one endpoint in common.
  \begin{equation}
    (e_{kl})R = \left \{
    \begin{array}{cl}  
      q^2 e_{kl} & \textrm{identical}\\
      e_{kl} & \textrm{noncrossing}\\
      e_{kl} & \textrm{to the left}\\
      (q^2-q) e_{ij} + q e_{kl} + (1-q) e_\third  & \textrm{to the right}\\
    \end{array}
    \right.
  \end{equation}
\end{exmp}

\begin{figure}
  \begin{tikzpicture}[scale=.6]
    \begin{scope}[xshift=-7cm]
      \coordinate (1) at (0,0);
      \coordinate (2) at (3.5,5.46);
      \coordinate (3) at (7,1);
      \coordinate (4) at (5.4,-1.2);
      \draw[GreenPoly] (1)--(2)--(3)--(4)--cycle;
      \draw[black] (1)--(2) node [midway,above left,black] {$a$};
      \draw[black,dashed] (1)--(3) node [midway,above left,black] {$b$};
      \draw[RedLine] (1)--(4) node [midway,below,black] {$c$}; 
      \draw[RedLine] (2)--(3) node [midway,above right,black] {$d$};
      \draw[BlueLine] (2)--(4) node [midway,right,black] {$e$};
      \draw[black] (3)--(4) node [midway,right,black] {$f$};
      \fill (1) circle (.6mm) node[anchor=east] {$p_1$};
      \fill (2) circle (.6mm) node[anchor=south] {$p_2$};
      \fill (3) circle (.6mm) node[anchor=west] {$p_3$};
      \fill (4) circle (.6mm) node[anchor=north] {$p_4$};
    \end{scope}

    \begin{scope}[shift={(3cm,1cm)}]
      \draw[|-angle 90,very thick] (-1,0) -- node [above] {$R^{24}_{\rc(24)}$} (1,0);
    \end{scope}

    \begin{scope}[xshift=5cm]
      \coordinate (1) at (0,0);
      \coordinate (2) at (3.12,6.752);
      \coordinate (3) at (6.62,2.292);
      \coordinate (4) at (5.4,-1.2);
      \draw[GreenPoly] (1)--(2)--(3)--(4)--cycle;
      \draw[black] (1)--(2) node [midway,above left,black] {$a'$};
      \draw[black,dashed] (1)--(3) node [midway,above left,black] {$b'$};
      \draw[RedLine] (1)--(4) node [midway,below,black] {$c'$}; 
      \draw[RedLine] (2)--(3) node [midway,above right,black] {$d'$};
      \draw[black] (2)--(4);
      \draw[BlueLine,dashed] (2)--(4) node [midway,right,black] {$e'$};
      \draw[black] (3)--(4) node [midway,right,black] {$f'$};
      \fill (1) circle (.6mm) node[anchor=east] {$p_1$};
      \fill (2) circle (.6mm) node[anchor=south] {$p'_2$};
      \fill (3) circle (.6mm) node[anchor=west] {$p'_3$};
      \fill (4) circle (.6mm) node[anchor=north] {$p_4$};
    \end{scope}
  \end{tikzpicture}
  \caption{The edge rescaling $R^{24}_{\rc(24)}$ which rescales the
    edge $e_{24}$ while fixing the edges $e_{23}$ and
    $e_{14}$.\label{fig:diagonal-rescale}}
\end{figure}
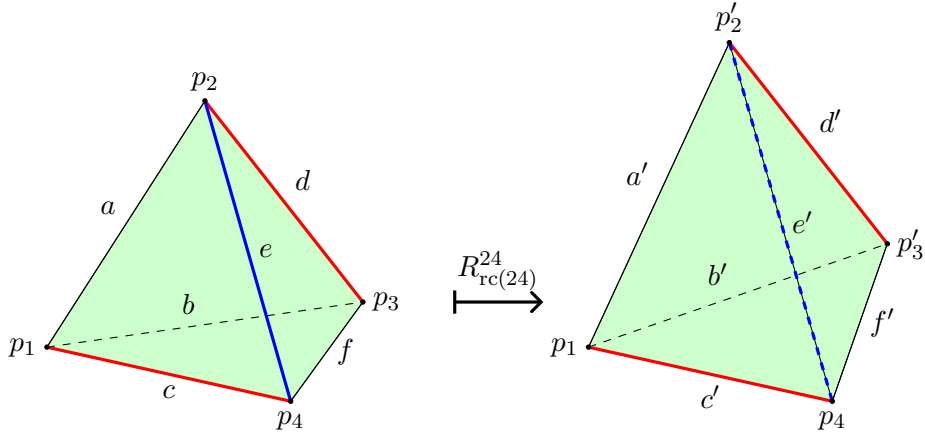

In order to extend Example~\ref{ex:boundary-edge} to more general
situations, we need one more computation.

\begin{prop}[Diagonal edges]\label{prop:diagonal-edge}
  Let $\Delta$ be a labeled euclidean tetrahedron.  If $\Delta'$
  is the labeled euclidean tetrahedron obtained by the edge rescaling
  $R = R^{24}_{\rc(24)}$, then the new edge norm $a'_{13}$ can be 
  computed from the edge norms of $\Delta$ as follows:
  \begin{equation*}
    a'_{13} = a_{13} + (q-1)^2 a_{24} + (q-1)(a_{14}+a_{23}) +(1-q)(a_{12} + a_{34}) 
  \end{equation*}
\end{prop}

\begin{proof}
  Since $v_{13} = v_{12} + v_{24} + v_{43}$, we have $v'_{13} = v_{12}
  + q v_{24} + v_{43}$.  Expanding $a'_{13} = \langle v'_{13}, v'_{13}
  \rangle$ we find that $a'_{13}$ is equal to $a_{12} + q^2 a_{24} +
  a_{34} + 2q \langle v_{12},v_{24} \rangle + 2q \langle v_{24},v_{43}
  \rangle + 2 \langle v_{12}, v_{34} \rangle$.  Using
  Proposition~\ref{prop:inner-prods} we find that $\langle
  v_{12},v_{24} \rangle = a_{14} - a_{12} - a_{24}$, $\langle
  v_{24},v_{43} \rangle = a_{23} - a_{24} - a_{34}$ and $\langle
  v_{12},v_{43} \rangle = a_{13} + a_{24} - a_{14} -
  a_{23}$. Substituting and simplifying yields the result.
\end{proof}

\begin{exmp}[Rescaling a diagonal edge]\label{ex:diagonal-edge}
  The row description for the rescaling matrix $R^{ij}_{\rc(ij)}$ is
  essentially identical to the one listed in
  Example~\ref{ex:boundary-edge}.  In particular, the formulas for the
  cases where $e_{kl}$ is identical to, noncrossing, to the left or to
  the right of $e_{ij}$ are the same as before.  The final geometric
  configuration that is possible when $e_{ij}$ is not a boundary edge
  is that $e_{ij}$ and $e_{kl}$ might be crossing.  When this happens
  $(e_{kl}) R^{ij}_{\rc(ij)}$ can be computed using
  Proposition~\ref{prop:diagonal-edge}.  If the clockwise ordering of
  $i$, $j$, $k$ and $l$ is $(k,i,l,j)$ then the answer is $e_{kl} +
  (q-1)^2 e_{ij} + (q-1) e_{kj} + (q-1) e_{il} + (1-q) e_{ki} + (1-q)
  e_{lj}$.  A more intrinsic geometric description would use the
  convex hull of $\{p_i,p_j,p_k,p_l\}$ in $\disc_n$.  The answer
  obtained is $e_{kl}$ plus $(q-1)^2 e_{ij}$ plus $(q-1)$ times the
  two boundary edges of the convex hull which are simultaneously to
  the right of $e_{ij}$ and to the left of $e_{kl}$ plus $(1-q)$ times
  the two boundary edges which are simultaneously to the right of
  $e_{ij}$ and to the left of $e_{kl}$.  A row description for
  $R^{ij}_{\lc(ij)}$ for arbitrary $i$ and $j$ can be computed in a
  similar fashion.
\end{exmp}

\section{Matrices}\label{sec:matrices}
In this section, we define three explicit representations of the braid
group and then we prove our main results.  The first representation,
and the most complicated, is the Lawrence-Krammer-Bigelow or LKB
representation.

\begin{defn}[LKB representation]\label{def:LKB-rep} 
  Let $q$ and $t$ be nonzero positive real numbers, let $\E$ be the
  set $\{e_{ij}\}$ with $1 \leq i < j \leq n$ of size $N =
  \binom{n}{2}$ and let $\R^N$ be the $N=\binom{n}{2}$-dimensional
  real vector space with $\E$ as its ordered basis.  The \emph{LKB
    representation} of the braid group is the map $\rho: \braid_n \to
  GL_N(\R)$ defined by the following action (from the right) of the
  standard braid group generators $s_{ij}$ (with $j=i+1$ and $1 \leq i
  <n$) on elements of $\E$.
  \begin{equation}
    (e_{kl})\rho(s_{ij}) = \left \{
    \begin{array}{cl}  
      tq^2 e_{kl} & i=k, j=l\\
      e_{kl} & i,j \notin \{k,l\} \\
      e_{jl} & i = k, j < l \\
      e_{kj} & i=l \\
      t(q^2-q)e_{ij} + q e_{ki} + (1-q) e_{kl}  & k < i, j = l \\
      (q^2-q) e_{ij} + q e_{il} + (1-q) e_{kl}  & j = k \\
    \end{array}
    \right.
  \end{equation}
\end{defn}

We make two remarks about this definition.

\begin{rem}[Left/right actions]\label{rem:left-right}
  In the literature, this action is written as an action from the left
  but our alteration only has the effect of transposing the relevant
  matrices.  Our choice is dictated by our desire to match up (a
  simplified version of) this representation with the obviously very
  similar edge rescaling matrices discussed in the previous section.
\end{rem}

\begin{rem}[The $t$ variable and its sign]\label{rem:t}
  The $t$ variable depends on the linear ordering of the vertices, it
  is associated the standard presentation of the braid group, and its
  presence obscures the fundamentally cyclically symmetric nature of
  the dependence on $q$.  To highlight this cyclic symmetry, we shall
  consider the specialization with $t=1$ below.  We should also note,
  however, that there are inconsistencies in the literature regarding
  the sign of $t$.  The variable $t$ in \cite{Kr00} corresponds to
  $-t$ in \cite{Bi01} (with an additional sign correction in
  \cite{Bi03}) and \cite{ItoWiest}.  We have written the LKB
  representation using Krammer's sign convention.  If we had followed
  Bigelow's we would be setting $t$ equal to $-1$.
\end{rem}

In light of our first main theorem, we call the simplified LKB
representation the \emph{simplicial representation} of the braid
group.

\begin{defn}[Simplicial representation] \label{def:simp-rep} 
  The \emph{simplicial representation} of the braid group is the
  specialization of the LKB representation with $t$ set equal to $1$.
  Concretely, let $q$ be a nonzero positive real number, let $\E$ be
  the set $\{e_{ij}\}$ with $1 \leq i < j \leq n$ in lexicographic
  order and let $\R^N$ be the $N=\binom{n}{2}$-dimensional real vector
  space with $\E$ as its ordered basis.  The \emph{simplicial
    representation} of the braid group is defined by the following
  action (from the right) of the standard braid group generators
  $s_{ij}$ (with $j=i+1$ and $1 \leq i <n$) on elements of $\E$.  We
  write $S_\sigma$ for the matrix that represents $s_\sigma$ with
  respect to the ordered basis $\E$ and we have introduced the
  notation $e_{\third}$ to denote the third side of the triangle when
  $e_{ij}$ and $e_{kl}$ have exactly one endpoint in common as in the
  previous section.
  \begin{equation}
    (e_{kl})S_{ij} = \left \{
    \begin{array}{cl}  
      q^2 e_{kl} &  i=k, j=l\\
      e_{kl} & i,j \notin \{k,l\}\\
      e_{\third} & i = k, j < l \\
      e_{\third} & i=l \\
      (q^2-q) e_{ij} + q e_{\third} + (1-q) e_{kl} & k < i, j = l \\
      (q^2-q) e_{ij} + q e_{\third} + (1-q) e_{kl} & j = k \\
    \end{array}
    \right.
  \end{equation}
  Now that the $t$ variable has been eliminated, some of the rows are
  identical and they can be rewritten more elegantly using the
  language of Definition~\ref{def:edge-pairs}.
  \begin{equation}
    (e_{kl})S_{ij} = \left \{
    \begin{array}{cl}  
      q^2 e_{kl} &  \textrm{identical}\\
      e_{kl} & \textrm{noncrossing}\\
      e_{\third} & \textrm{to the left} \\
      (q^2-q) e_{ij} + q e_{\third} + (1-q) e_{kl} & \textrm{to the right} \\
    \end{array}
    \right.
  \end{equation}
\end{defn}

Our third representation is obtained by also eliminating the $q$ variable.

\begin{defn}[Permutation representation]\label{def:perm-rep} 
  The \emph{permutation representation} of the braid group that we are
  interested in is the one obtained from the simplicial representation
  by setting $q=1$ (or both $t = q =1$ in the LKB representation).
  This encodes the permutation of the edges induced by the
  corresponding permutation of the vertices.  We write $P_\sigma$ for
  the matrix corresponding to $s_\sigma$.  Its row description is as
  follows.
  \begin{equation}
    (e_{kl})P_{ij} = \left \{
    \begin{array}{cl}  
      e_{kl} &  \textrm{identical, crossing or noncrossing}\\
      e_{\third} & \textrm{to the left or right} \\
    \end{array}
    \right.
  \end{equation}
\end{defn}

It should be clear that the simplicial representation matrix $S_{ij}$
is very closely connected, but not quite identical to the rescaling
matrix $R^{ij}_{\rc(ij)}$ given in Example~\ref{ex:boundary-edge}.
The difference is the permutation matrix $P_{ij}$.

\begin{exmp}[Geometry of $S_{12}$]\label{ex:s12}
  The matrices corresponding to the first standard generator $s_{12}$
  of the four string braid group in the simplicial representation and
  the permutation representation are as follows:
  \begin{equation}
    S_{12} = \begin{bmatrix} 
      q^2 & 0 & 0 & 0 & 0 & 0 \\ 
      0 & 0 & 0 & 1 & 0 & 0 \\ 
      0 & 0 & 0 & 0 & 1 & 0\\ 
      q^2-q & q & 0 & 1-q & 0 & 0\\ 
      q^2-q & 0 & q & 0 & 1-q & 0 \\ 
      0 & 0 & 0 & 0 & 0 & 1
    \end{bmatrix}
    \hspace{1em}
    P_{12} = \begin{bmatrix} 
      1 & 0 & 0 & 0 & 0 & 0 \\ 
      0 & 0 & 0 & 1 & 0 & 0\\ 
      0 & 0 & 0 & 0 & 1 & 0 \\ 
      0 & 1 & 0 & 0 & 0 & 0\\ 
      0 & 0 & 1 & 0 & 0 & 0 \\ 
      0 & 0 & 0 & 0 & 0 & 0 \\
    \end{bmatrix}
  \end{equation}
  It is now straightforward to check that $S_{12} = P_{12}
  R^{12}_{234} = R^{12}_{134} P_{12}$.  The matrix $R^{12}_{234} =
  R^{12}_{\rc(12)}$ is listed explicitly in Example~\ref{ex:tetra} and
  the row description of $R^{12}_{134} = R^{12}_{\lc(12)}$ is given in
  Example~\ref{ex:boundary-edge}.
\end{exmp}

Before proceeding to the proofs of our main results, we introduce one
final definition which makes our precise results easier to state.

\begin{defn}[Relabeling and rescaling]\label{def:relabel-rescale}
  Let $\sigma$ be a noncrossing permutation in $\sym_n$ and let
  $S_\sigma$ be the explicit matrix representing $s_\sigma$ under the
  simplicial representation.  We say that $S_\sigma$ \emph{relabels
    and rescales} if $S_\sigma = P_\sigma R^\sigma_{\rc(\sigma)} =
  R^\sigma_{\lc(\sigma)} P_\sigma$, where $\lc(\sigma)$ and
  $\rc(\sigma)$ are the left/right complements of $\sigma$.
\end{defn}

For clarity we also say what this definition means geometrically.  To
say that the matrix $S_\sigma$ relabels and rescales means that the
effect it has on labeled euclidean simplices (keeping in mind that we
multiply from right to left) is to first rescale the edges (fixing the
length and direction of the edges in the right complement of $\sigma$
while multiplying the lengths of the edges in $\sigma$ by a factor of
$q$) followed by the edge relabeling induced by the way $\sigma$
permutes the vertices.  Alternatively, the edge relabeling $P_\sigma$
can be performed first (i.e. on the right), in which case it is the
edges of the left complement of $\sigma$ whose length and direction
are fixed while the lengths of the edges in $\sigma$ are multiplied by
a factor of $q$.  In this language, Example~\ref{ex:s12} shows that
the matrix $S_{12}$ in the simplicial representation of $\braid_4$
relabels and rescales.  The next proposition shows that the standard
generators in all of the braid groups share this property.

\begin{prop}[Standard generators]\label{prop:std-gen}
  For every standard generator $s_{ij}$ of the braid group $\braid_n$,
  the corresponding matrix $S_{ij}$ in the simplicial representation
  relabels and rescales.
\end{prop}

\begin{proof}
  This is essentially immediate at this point once we compare the row
  description of $S_{ij}$ in Definition~\ref{def:simp-rep} with the
  row descriptions of $R^{ij}_{\rc(ij)}$ and $R^{ij}_{\lc(ij)}$ in
  Example~\ref{ex:boundary-edge} and note that multiplying by $P_{ij}$
  on the left switches the rows to the left of $e_{ij}$ with the rows
  to the right of $e_{ij}$ (which has the effect of switching which
  edge is denoted $e_{kl}$ and which is $e_{\third}$), while
  multiplying by $P_{ij}$ on the right permutes columns and thus the
  subscripts on the $e$'s that occur in the various terms of the row
  descriptions.
\end{proof}

From Proposition~\ref{prop:std-gen} we deduce our first result.

\renewcommand{\themain}{\ref{main:braid}}
\begin{main}[Braids reshape simplices]
  The simplicial representation of the $n$-string braid group
  preserves the set of $\binom{n}{2}$-tuples of positive reals that
  represent the squared edge lengths of a nondegenerate euclidean
  simplex with $n$ labeled vertices.
\end{main}

\begin{proof}
  First note that it suffices to prove that this holds for some
  generating set of $\braid_n$.  Next, both vertex relabelings and
  edge rescalings clearly preserve the set of $\binom{n}{2}$-tuples
  that describe the squared edge lengths of a nondegenerate euclidean
  simplex with $n$ labeled vertices, so Proposition~\ref{prop:std-gen}
  completes the proof.
\end{proof}

After we proved Theorem~\ref{main:braid}, we discovered that it had
already been established (in a different language) in the dissertation
of Arkadius Kalka \cite{Ka07}.  Since our proof is more geometric in
nature and seemingly simpler than his, we believe that our proof is of
independent interest.  Also, he does not appear to have established
anything resembling our Theorem~\ref{main:simple}.  Once we know that
some dual simple braids relabel and rescale, it is straightforward to
show that their products (at least those which remain dual simple
braids) do so as well.

\begin{prop}[Products]\label{prop:products}
  Let $\sigma_1$ and $\sigma_2$ be noncrossing permutations in
  $\sym_n$ such that $s_{\sigma_1}$, $s_{\sigma_2}$ and
  $s_{\sigma_1\sigma_2}$ are dual simple braids.  If both
  $S_{\sigma_1}$ and $S_{\sigma_2}$ relabel and rescale, then
  $S_{\sigma_1 \sigma_2}$ relabels and rescales.
\end{prop}

\begin{proof}
  One consequence of the fact that $s_{\sigma_1}$, $s_{\sigma_2}$ and
  $s_{\sigma_1\sigma_2}$ are dual simple braids is that there are
  noncrossing permutations $\sigma_3$, $\sigma_4$ and $\sigma_5$ such
  that $\delta = \sigma_1 \sigma_2 \sigma_3 = \sigma_1 \sigma_4
  \sigma_2 = \sigma_5 \sigma_1 \sigma_2$
  (Definition~\ref{def:five-perms}).  For $S_{\sigma_1 \sigma_2} =
  S_{\sigma_1} S_{\sigma_2}$ we have the following equalities.
  \begin{eqnarray*}
    S_{\sigma_1} S_{\sigma_2} &=& (P_{\sigma_1} R^{\sigma_1}_{\sigma_4
      \sigma_2}) (R^{\sigma_2}_{\sigma_1\sigma_4} P_{\sigma_2})\\
    &=& P_{\sigma_1} R^{\sigma_1\sigma_2}_{\sigma_4} P_{\sigma_2}\\
    &=& P_{\sigma_1} P_{\sigma_2}  R^{\sigma_1\sigma_2}_{\sigma_3}\\
    &=& R^{\sigma_1\sigma_2}_{\sigma_5} P_{\sigma_1} P_{\sigma_2}
  \end{eqnarray*}
  The first line uses the hypotheses on $S_{\sigma_1}$ and
  $S_{\sigma_2}$.  The second combines the two edge scalings into a
  single rescaling.  In particular, the rescaling
  $R^{\sigma_1}_{\sigma_4 \sigma_2}$ rescales the edges in $\sigma_1$
  by $q$ and fixes the edges in $\sigma_4$ and $\sigma_2$, as well as
  the rest of the edges in the product $\sigma_4 \sigma_2$.
  Similarly, the rescaling $R^{\sigma_2}_{\sigma_1 \sigma_4}$ fixes
  the edges in in $\sigma_1$ and $\sigma_4$ and rescales those in
  $\sigma_2$ by a factor of $q$.  Thus, in the product of these two
  edge rescalings, the edges in $\sigma_1$ and $\sigma_2$ are rescaled
  by $q$ and those in $\sigma_4$ are fixed.  The third and fourth
  lines simply conjugate the points involved.  Since $P_{\sigma_1}
  P_{\sigma_2} = P_{\sigma_1 \sigma_2}$ and $\sigma_5$ and $\sigma_3$
  are the left/right complements of $\sigma_1 \sigma_2$, this
  completes the proof.
\end{proof}

Note that the same equalities used in the proof, slightly rearranged,
would show that if any two of $S_{\sigma_1}$, $S_{\sigma_2}$ and
$S_{\sigma_1\sigma_2}$ relabel and rescale then so does the third.
Propositions~\ref{prop:std-gen} and~\ref{prop:products} are not quite
enough to prove our second main result because not all dual simple
braids are products of standard generators.  We need to extend
Proposition~\ref{prop:std-gen} to the full set of dual generators.

\begin{prop}[Dual generators]\label{prop:dual-gen}
  For every dual generator $s_{ij} \in \gen_n$ of the braid group
  $\braid_n$, the corresponding matrix $S_{ij}$ in the simplicial
  representation relabels and rescales.
\end{prop}

\begin{proof}
  An explicit description of the matrix for $s_{ij}$ under the LKB
  representation is given in Krammer's earlier paper \cite{Kr00}.  If
  we set $t=1$ in that description, we find that the matrix $S_{ij}$
  has the exact same description as it does when $j=i+1$ as given in
  Definition~\ref{def:simp-rep} except that a new case must be added
  that gives the result of $(e_{kl}) S_{ij}$ when $e_{ij}$ and
  $e_{kl}$ cross.  This simplification of the answer listed in
  \cite{Kr00} agrees with the corresponding row of $P_{ij}
  R^{ij}_{\rc(ij)}$, which is the same as the corresponding row of
  $R^{ij}_{\lc(ij)} P_{ij}$ obtained by combining
  Example~\ref{ex:diagonal-edge} and Definition~\ref{def:perm-rep}.
\end{proof}

Our second main result is an immediate corollary.

\renewcommand{\themain}{\ref{main:simple}}
\begin{main}[Dual simple braids relabel and rescale]
  Under the simplicial representation of the braid group, each dual
  simple braid relabels and rescales.  Concretely, for each $\sigma
  \in NC_n$, we have $S_\sigma = P_\sigma R^\sigma_{\rc(\sigma)} =
  R^\sigma_{\lc(\sigma)} P_\sigma$.
\end{main}

\begin{proof}
  Proposition~\ref{prop:dual-gen} shows that the assertion is true for
  each dual generator.  Then Proposition~\ref{prop:products} and a
  simple induction extends this fact to all of the dual simple braids.
\end{proof}

\section{Final remarks}\label{sec:remarks}

In this final section we make a few remarks about the origins of these
results and some promising directions for futher investigation.  

To explain how we stumbled upon this point of view, we first need to
review the differences between the three papers by Krammer and Bigelow
establishing the linearity of the braid groups.  In his earlier
article Daan Krammer used a ping-pong type argument on the dual
Garside structure of the $4$~string braid group to establish that the
LKB representation representation is faithful for suitably generic
values of $q$ and $t$ \cite{Kr00}.  Next Stephen Bigelow replaced
Krammer's algebraic approach with a more topological one and succeeded
in showing that these representations are faithful for all $n$
\cite{Bi01}.  In his later article Krammer used an alternative version
of his original approach that also succeeded in establishing linearity
for every $n$ \cite{Kr02}.  The main difference between the two papers
by Krammer is that the first uses the dual Garside presentation of the
braid groups (indexed by the $q$ variable and closely related to the
Birman-Ko-Lee presentation introduced in \cite{BiKoLe98}) while the
second reverts to the standard presentation of the braid group
(indexed by the $t$ variable).

Shortly after these articles appeared, their results were extended to
prove linearity results for various other Artin groups, all based more
or less on the approach used in Krammer's second article \cite{Kr02}.
Fran\c{c}ois Digne extended linearity to the Artin groups of
crystallographic type \cite{Di03} and Arjeh Cohen and David Wales
proved linearity for all spherical Artin groups \cite{CoWa02}. Finally
Luis Paris generalized these results further by proving that the Artin
\emph{monoid} is linear for all Artin groups \cite{Pa02}.  In fact,
all of these proofs establish the linearity of the positive monoid.
When the Artin group is spherical, the Artin group is the group of
fractions of the positive monoid and thus linearity of the positive
monoid implies linearity for the group.  

For general Artin groups it is known that the positive monoid
generated by the standard minimal generating set is not large enough
for this implication to hold.  On the other hand, the positive monoid
generated by the larger dual generating set might have this property.
This led us to attempt to generalize Krammer's initial argument using
the dual generating set.  If one could find a proof of linearity for
all the braid groups focused on the $q$ variable as in \cite{Kr00},
then there is a chance that it might generalize as in the papers by
Digne, Cohen and Wales and Paris to eventually produce linearity
results for non-spherical Artin groups.  This was our initial
motivation.

We first wrote code in \texttt{sage} to investigate the properties of
the LKB matrices and found, experimentally, nice ways to decompose
them and we began to isolate the changes that each factor was making.
The interpretation of the simplified version as modifications of edge
norms of simplices was one of the final steps in our evolving
understanding of these representations.  

We conclude this article with three directions for additional
research.

\begin{prob}[Linearity]
  The simplified simplicial representation is known not to be faithful
  for large $n$ because it is also known as the symmetric tensor
  square of the Burau representation (which is known to not be
  faithful for $n\geq 5$ \cite{Bi99}).  Thus, establishing a new $q$
  based proof of braid group linearity requires the variable $t$ to
  remain generic.  One project is to find a way to extend the
  geometric interpretation given here so that the $t$ variable can
  remain generic and then to use this extended interpretation to give
  an alternative proof of braid group linearity focused on the $q$
  variable.  In particular, one should try to show directly that the
  dual positive monoid generated by full dual generating set acts
  faithfully using a ping-pong type argument similar to the one in the
  original Krammer article on $\braid_4$ \cite{Kr00}.
\end{prob}

\begin{prob}[Dual Garside length]
  The $q$~variable in the LKB representation has received very little
  attention since Krammer's earlier paper primarily because it was the
  $t$ variable which was the focus for the more general proof.
  Recently, however, Tetsuya Ito and Bert Wiest posted an article that
  proves one of the facts originally conjectured by Krammer in
  \cite{Kr00}, namely, that the highest power of $q$ in the LKB
  representation of a dual positive braid is twice its dual Garside
  length \cite{ItoWiest}.  It seems quite likely that one could give
  an alternative and elementary proof of their results using the
  geometric understanding of the $q$~variable introduced in this
  article.
\end{prob}

\begin{prob}[Spherical Artin groups]
  The set of labeled euclidean simplices, with dilated simplices
  identified, is one of the standard parameterizations of the higher
  rank symmetric space $SL(V)/SO(V)$ and the simplicial representation
  appears to act on this space by isometries.  Once this action is
  made explicit, it should be possible to define a similar
  construction and to give a similar interpretation for all of the
  spherical Artin groups once the focus on labeled euclidean simplices
  is replaced by linear transformations of root systems.
\end{prob}

\providecommand{\bysame}{\leavevmode\hbox to3em{\hrulefill}\thinspace}
\providecommand{\MR}{\relax\ifhmode\unskip\space\fi MR }
\providecommand{\MRhref}[2]{%
  \href{http://www.ams.org/mathscinet-getitem?mr=#1}{#2}
}
\providecommand{\href}[2]{#2}

\end{document}